\documentclass[a4paper,11pt]{article}
\usepackage{amsmath,amsfonts,amssymb,graphics,amsthm}
\usepackage[dvips]{graphicx}
\usepackage{subfigure}
\usepackage{natbib}
\usepackage{a4wide}

\newtheorem{thm}{Theorem} 
\newtheorem{lemma}{Lemma}

\theoremstyle{definition}\newtheorem{defn}{Definition}
\newtheorem{rmk}{Remark}

\def\tmix{t_{\mathrm{mix}}}

\newcommand{\xth}{^\text{th}}
\newcommand{\tc}{\tau_{\mathrm{couple}}}

\def\cR{\mathcal{R}}

\def\cL{\mathcal{L}}

\def\cC{\mathcal{C}}

\def\tmix{t_{\mathrm{mix}}}
\def\tv{_{\mathrm{TV}}}

\def\P{\mathbb{P}}
\def\E{\mathbb{E}}
\def\eps{\varepsilon}

\DeclareMathOperator{\var}{Var}

\newcommand{\indic}[1]{\mathbf{1}_{\{#1\}}}
\parskip \medskipamount

\title{Partial mixing of semi-random transposition shuffles}
\author{Richard Pymar \thanks{LAREMA -- UMR CNRS 6093, Universit\'{e} d'Angers, 2 Boulevard Lavoisier, 49045 Angers cedex 01; Part of this work was done when the author was an EPSRC-funded PhD student at the University of Cambridge \texttt{richard.pymar@univ-angers.fr}}}
\date{\today}
\begin{document}
\maketitle

\begin{abstract}
We show that for any semi-random transposition shuffle on $n$ cards, the mixing time of any given $k$ cards is at most $n\log k$, provided $k=o((n/\log n)^{1/2})$.  In the case of the top-to-random transposition shuffle we show that there is cutoff at this time with a window of size $O(n)$, provided further that $k\to\infty$ as $n\to\infty$ (and no cutoff otherwise).  For the random-to-random transposition shuffle we show cutoff at time $(1/2)n\log k$ for the same conditions on $k$. Finally, we analyse the cyclic-to-random transposition shuffle and show partial mixing occurs at time $\le\alpha n\log k$ for some $\alpha$ just larger than 1/2.  We prove these results by relating the mixing time of $k$ cards to the mixing of one card. Our results rely heavily on coupling arguments to bound the total variation distance.
\end{abstract}

\section{Introduction}

We denote by $S_n$ the symmetric group on $n$ elements which we shall view as a deck of $n$ labeled cards. We use $[n]$ to denote $\{1,\ldots,n\}$. We are interested in a notion of mixing of these $n$ cards when the shuffling mechanism is semi-random transposition shuffling. At every step of such a shuffle we choose 2 cards -- one with our left hand and one with our right (independently of each other and of the past choices). We then switch the positions of the two selected cards.  We denote by $L_t$ and $R_t$ the location chosen by the left hand and right hand at time $t$, respectively.  The left hand is allowed to choose a card according to any rule, deterministic or stochastic, and it may also depend on time.  The right hand chooses a card uniformly.

As well as stating a result for general semi-random transposition shuffles, we also focus on 3 particular shuffles. These are the cyclic-to-random transposition shuffle, with $L_t=t \mod n$, the top-to-random transposition shuffle, with $L_t=1$, and the random-to-random transposition shuffle, with $L_t\sim$Unif($[n]$).


We first state our results and then present some background and motivation.  We shall denote by $\mu^\sigma_{c_1,\ldots,c_k}$ the joint law of the locations of cards $c_1,\ldots,c_k$ in a permutation $\sigma$. Let $\sigma_t$ be a permutation at time $t$ which starts at time 0 from permutation $\sigma_0$ and evolves by a given semi-random transposition shuffle. Let $\pi$ be a uniform permutation.  We define $d_{c_1,\ldots,c_k}(t)$ by
\[
d_{c_1,\ldots,c_k}(t):=\max_{\sigma_0}\|\mu^{\sigma_t}_{c_1,\ldots,c_k}-\mu^\pi_{c_1,\ldots,c_k}\|\tv
,\]
where $\|\mu-\nu\|\tv$ is the total variation distance between two measures $\mu$ and $\nu$.
Let $\Omega_k$ denote the set of all subsets of $[n]$ of size $k$.
\begin{defn}
We define the $k$-\emph{partial mixing time} (at level $\eps$) of a shuffle to be \[
  \tmix^k(\eps):=\min\Big\{t\ge0:\,\max_{\{c_1,\ldots,c_k\}\in\Omega_k}d_{c_1\ldots,c_k}(t)<\eps\Big\}.
  \]
\end{defn}

Our main result is the following.

\begin{thm}\label{T:main}Suppose $k=o((n/\log n)^{1/2})$ and fix $\eps>0$. For any semi-random transposition shuffle with $k$-partial mixing time $\tmix^k(\eps)$ and any $\delta>0$, there exists $n_0=n_0(\delta)$ such that for all $n>n_0$,
\begin{flalign*}
(i).\quad&\tmix^k(\eps)\le \tmix^1((\eps-\delta)/k),\quad\text{and moreover,}&\\
(ii).\quad&\tmix^k(1/4)\le n(\log k+3/2).&
\end{flalign*}

\end{thm}

Theorem \ref{T:main} allows us to calculate an upper bound on the $k$-partial mixing time of any semi-random transposition shuffle by just considering the movement of one card in the deck.

\begin{thm}\label{cyclicupper}
 Suppose $k=o((n/\log n)^{1/2})$. For the cyclic-to-random transposition shuffle, there exists a constant $C>0$, such that for all $n$ sufficiently large, \[
 \tmix^k(1/4)\le 0.5006n(\log k+C)
 .\]
\end{thm}

\begin{defn}
  We say that a shuffle has \emph{$k$-partial cutoff at time $\tmix$ with a window of size $\omega(n)=o(\tmix)$} if the following two conditions hold:
\begin{flalign*}
  (i).\quad &\lim_{\alpha\to\infty}\limsup_{n\to\infty}\max_{(c_1,\ldots,c_k)\in \Omega_k}d_{c_1,\ldots,c_k}(\tmix+\alpha\omega(n))=0,&\\
  (ii).\quad &\lim_{\alpha\to -\infty}\liminf_{n\to\infty}\max_{(c_1,\ldots,c_k)\in \Omega_k}d_{c_1,\ldots,c_k}(\tmix+\alpha\omega(n))=1.&
\end{flalign*}
\end{defn}

\begin{thm}\label{T:top}Suppose $k=o((n/\log n)^{1/2})$ and also $k\to\infty$ as $n\to\infty$.
Then the top-to-random transposition shuffle has $k$-partial cutoff at time $n\log k$ with a window of size $O(n)$. Furthermore, if $k<K$, for some constant $K$ for all $n$, then there is no cutoff.
\end{thm}

\begin{thm}\label{T:random}
Suppose $k=o((n/\log n)^{1/2})$ and also $k\to\infty$ as $n\to\infty$.
Then the random-to-random transposition shuffle has $k$-partial cutoff at time $0.5n\log k$ with a window of size $O(n)$. Furthermore, if $k<K$, for some constant $K$ for all $n$, then there is no cutoff.\end{thm}

The study of mixing times of Markov chains and the search for cutoff is a much-studied area of probability. In terms of walks on the symmetric group, random transpositions are one of the most natural and simplest models. It was first shown by \cite{shah} that the random-to-random transposition shuffle has cutoff at time $0.5n\log n$ with a window of size $O(n)$.  Their technique uses Fourier analysis on the symmetric group. Since then, this result has been shown with a strong-stationary time argument by \cite{matthews} and with coupling arguments by \cite{kcycle} and independently by \cite{blum}.

Theorem \ref{T:main} gives an upper bound of $n\log k$ on the $k$-partial mixing time of any semi-random transposition shuffle.  Regarding full mixing, it has been shown by \cite{saloff} and independently by \cite{ganap} that the mixing time of any semi-random transposition shuffle is at most $n\log n$, with the top-to-random having cutoff at this time.

The cyclic-to-random transposition shuffle was invented by \cite{thorp} and the question of its mixing time was posed by \cite{aldous}. Unlike with the top-to-random and random-to-random shuffles, after just $n$ steps of the cyclic-to-random shuffle, every card has almost surely been selected with either the left hand or the right hand at least once.  The standard coupon-collector argument for establishing a lower bound does not give the correct order for the mixing time of this shuffle.  The current best known lower bound is about 0.12$n\log n$, shown by \cite{peresmoss} obtained by analysing the eigenfunctions of the transition matrix for the movement of a single card.  The best known upper bound is the general $n\log n$ of any semi-random transposition shuffle. It remains an open question to determine if and when cutoff occurs for this shuffle.

The notion of studying the evolution of only some of the cards in a deck has been previously considered by \cite{riffle}. Here they calculate the distribution of the location of one card after a step of riffle shuffling. For the cyclic-to-random transposition shuffle the distribution of the location of one card after a round ($n$ steps) has been calculated by \cite{pinsky} and independently by \cite{pymar}. It is a natural extension to try to understand how $k$ cards in a deck evolve and our results reveal interesting dynamics.  As with full mixing, the top-to-random shuffle is the slowest shuffle in terms of partial mixing, and using the random-to-random shuffle halves this time.  We also find that the cyclic-to-random does not take much longer to mix $k$ cards than the random-to-random (and may even be much faster).

\emph{Structure of the rest of the paper:} In section \ref{S:main}, we give the proof of Theorem \ref{T:main} by first studying the movement of one card and then the joint movement of a given $k$ cards.  Next, in section \ref{S:cutoff} we show partial cutoff of the top-to-random and random-to-random shuffles by using Theorem \ref{T:main}. We then give our upper bound of the $k$-partial mixing time of the cyclic-to-random shuffle in section \ref{S:cyclic}, before finally presenting some open questions in section \ref{S:further}.



\section{Proof of Theorem \ref{T:main}}\label{S:main}

\begin{defn}
For each non-negative integer $t$, let $\ell_t$ be a distribution on $[n]$.  We say that a permutation evolves by \emph{$(\ell_t)_{t\ge0}$-to-random} if it evolves according to a semi-random transposition shuffle and for each $t$ the choice of the left hand at time $t$ is chosen according to the distribution $\ell_t$.
\end{defn}

We begin by studying the mixing time of one card for any $(\ell_t)_{t\ge0}$-to-random transposition shuffle.  Let $(\sigma_t)_{t\ge0}$ denote the state at time $t$ of a permutation that evolves by $(\ell_t)_{t\ge0}$-to-random, started from permutation $\sigma_0$, and $(\pi_t)_{t\ge0}$ a permutation also evolving by $(\ell_t)_{t\ge0}$-to-random, started from a uniform permutation $\pi_0$. We denote by $\sigma^{-1}(i)$ the location occupied by the card with label $i$ in a permutation $\sigma$. Recall that for each $t$, $L_t$ and $R_t$ are independent from each other and from past choices of the left and right hands.

\begin{lemma}\label{onecard} For every $(\ell_t)_{t\ge0}$-to-random transposition shuffle,
 \[
\max_{i\in[n]}d_i(t):=\max_{i\in [n]}\max_{\sigma_0}\| \mu_i^{\sigma_t}-\mu_i^{\pi_t}\|\tv \le \exp(-t/n).\]
\end{lemma}
\begin{proof}
We use the coupling description of the total variation distance, that is,\[
\|\mu-\nu\|_{\mathrm{TV}}=\inf\{\P(X\neq Y):\,(X,Y)\text{ is a coupling of $\mu$ and $\nu$}\}
.\]

For each $t$, we use the same $L_t$ for both processes $\sigma_t$ and $\pi_t$.  We denote by $t-$ the time just before the $t\xth$ transposition.

To describe the joint evolution of $\sigma_t$ and $\pi_t$, we choose at time $t$ the card in location $R_t$ with the right hand in both processes unless $R_t\in\{\sigma_{t-}^{-1}(i),\pi_{t-}^{-1}(i)\}$.  If $R_t=\sigma_{t-}^{-1}(i)$, we select position $\sigma_{t-}^{-1}(i)$ in deck $\pi_t$ and select position $\pi_{t-}^{-1}(i)$ in deck $\sigma_t$ (so that card $i$ is not chosen in either deck).  On the other hand, if $R_t=\pi_{t-}^{-1}(i)$ we select position $\pi_{t-}^{-1}(i)$ in $\pi_t$ and position $\sigma_{t-}^{-1}(i)$ in $\sigma_t$ (so that card $i$ is chosen in both decks). It is clear that the marginal distributions are correct -- in each deck the probability of selecting a particular card with the right hand is $1/n$, independently of all past choices.

The result of this coupling is that once card $i$ is chosen with the right hand in $(\pi_t)$, the locations of card $i$ in both processes will become and then stay forever equal.  The time to select card $i$ in $(\pi_t)$ is Geom$(1/n)$. Thus we have \[
\P(\sigma_t^{-1}(i)=\pi_t^{-1}(i))\ge 1-\exp(-t/n)
,\]and the bound on the total variation distance follows.
\end{proof}

We refer to the $k$ cards of interest $c_1,\ldots,c_k$ as the \emph{special} cards.
We present now a bound on the number of times the left hand chooses one of the special $k$ cards up to time $t$, for any $(\ell_t)_{t\ge0}$-to-random transposition shuffle. Choose such a shuffle and let $(L_t)_{t\ge0}$ be an independent sequence of random variables such that $L_t\sim \ell_t$ for each $t$.

\begin{lemma}\label{leftatspecial}
There exists a constant $C$ such that for any choice of $k$ special cards, for each $t\ge0$ and $n$ sufficiently large,
$$
\E\Big(\sum_{s=1}^t \indic{\exists i\in[k]:\, L_s=\sigma_{s-}^{-1}(c_i)}\Big)\le Ck\Big(\frac{t}{n}+\log t\Big).
$$
\end{lemma}
\begin{proof}
First note that we can write
\[
 \sum_{s=1}^t \indic{\exists i\in[k]:\, L_s=\sigma_{s-}^{-1}(i)} = \sum_{s=1}^t \sum_{i=1}^k \indic{L_s=\sigma_{s-}^{-1}(i)}
=\sum_{i=1}^k\sum_{s=1}^t \indic{L_s=\sigma_{s-}^{-1}(i)}.
\]
Let $N_i(t)$ be the number of times that the left hand selects card $i$ by time $t$, that is,
\[
 N_i(t)=\sum_{s=1}^t \indic{L_s=\sigma_{s-}^{-1}(i)}.
\]
Let $\tau_i(m)$ be the time between the $(m-1)\xth$ and the $m\xth$ selection of card $i$ with the left hand, that is,\[
\tau_i(m):=\min\big\{t:\,t>\sum_{\ell=1}^{m-1}\tau_i(\ell),\,L_t=\sigma_{t-}^{-1}(i)\big\}-\sum_{\ell=1}^{m-1}\tau_i(\ell)
.\]  Then $N_i(t)$ satisfies
\begin{align}
\{N_i(t)\ge x\}=\big\{\sum_{m=1}^x \tau_i(m)\le t\big\}. \label{Nit}
\end{align}
Notice now that $\tau_i(m)\ge\tilde\tau_i(m)$ where $\tilde\tau_i(m)$ is the amount of time after the $(m-1)\xth$ selection of card $i$ with the left hand and the next time card $i$ is selected (with either the left or right hand), that is, \[
\tilde\tau_i(m):=\min\big\{t:\, t>\sum_{\ell=1}^{m-1}\tau_i(\ell),\,\sigma_{t-}^{-1}(i)\in\{L_t,R_t\}\big\}-\sum_{\ell=1}^{m-1}\tau_i(\ell)
.\] We further define $T_i(r)$ to be the time at which card $i$ is chosen with the left hand for the $r\xth$ time.  That is, $ T_i(r)=\sum_{m=1}^{r} \tau_i(m).$ Note that for $s$ satisfying

\[
 T_i(r)\le s < T_i(r)+\tilde\tau_i(r+1),
\]

we have $\sigma_{s-}^{-1}(i)=R_{T_i(r)},$ where $R_t$ is the location chosen with the right hand at time $t$.  We show that for each $m>1$, $\E(\tilde\tau_i(m))\ge c_1n$, for some constant $c_1$.

Firstly, at each time $s>T_i(m-1)$ we have probability $1/n$ of choosing card $i$ with the right hand.  On the other hand, to select card $i$ with the left hand it has to be in a location accessible by the left hand. For example, in the case of top-to-random, the left hand can only ever select the card at the top of the deck. Therefore in this case if at time $t$ card $i$ was in any position other than at the top of the deck, it would certainly be chosen with the right hand before the left hand.  Furthermore, since after the left hand selects card $i$ it is moved to a uniform location, in order to minimize the expected time to select this card again, the next $n$ positions selected by the left hand $L_{T_i(m-1)+1}, \ldots, L_{T_i(m-1)+n}$ should be a bijection of $[n]$.  We formalize this idea below.  We let \begin{align*}
H&=\inf\{h\ge1:\, L_{T_i(m-1)+h}=R_{T_i(m-1)}\},\\
R&=\inf\{r\ge1:\, R_{T_i(m-1)+r}=R_{T_i(m-1)}\}.
\end{align*}

 Note that $H$ and $R$ are independent and that $R$ is geometrically distributed with success probability $1/n$.  Further, in the case where $L_{T_i(m-1)+1}, \ldots, L_{T_i(m-1)+n}$ is a bijection of $[n]$, $H$ is uniformly distributed on $[n]$.

 Depending on the semi-random transposition shuffle we are using, the $(\tilde\tau_i(m))_{m\ge1}$ may not be independent. However, we claim that for each $m>1$, $\tilde\tau_i(m)$ stochastically dominates $\hat\tau_i(m)$, independent identically distributed random variables with distribution $\min(H,R)$.  We use a coupling argument to show this stochastic domination. We define $\hat\tau_i(m)$ to be the amount of time after $T_i(m-1)$ until the left hand or right hand chooses card $i$ but now with different $L$-values.  We denote these new $n$ $L$-values after time $T_i(m-1)$ as $\hat L_i^1(m-1),$ \mbox{$\hat L_i^2(m-1),\ldots, \hat L_i^n(m-1)$}.  This list is a bijection of $[n]$ with the property that if the first occurrence of the value $j$ in $L_{T_i(m-1)+1}, L_{T_i(m-1)+2}\ldots,$ appears before the first occurrence of $j'$, that is if \[
\min\{ t>0:\, L_{T_i(m-1)+t}=j\}<\min\{t>0:\, L_{T_i(m-1)+t}=j'\},
\]
then $j$ must appear before $j'$ in $\hat L_i^1(m-1), \ldots, \hat L_i^n(m-1)$.

We let $J=\{j\in [n]:\, \exists\, t>0:\, L_{T_i(m-1)+t}=j\}$.  For $j\in [n]$ which do not occur in $L_{T_i(m-1)+1}, L_{T_i(m-1)+2},\ldots$, i.e. $j\in [n]\setminus J$, we use the rule that they are listed arbitrarily in $\hat L_i^1(m-1), \ldots, \hat L_i^n(m-1)$ but after all values that do occur.  That is, \[
\forall\, j_1\in J, j_2\in [n]\setminus J,\,\mathrm{if}\, \hat L_i^{t_1}=j_1, \hat L_i^{t_2}=j_2, \,\mathrm{then}\, t_1<t_2.
\]

At time $T_i(m-1)$ card $i$ is moved to location $R_{T_i(m-1)}$ and will remain there until chosen with the left hand or the right hand.  The choice of $\hat L_i^j(m-1)$ for $j\ge 1$ ensures that $\hat\tau_i(m)\le \tilde\tau_i(m)$.  The independence and identical distribution properties of $\hat\tau_i(m)$ for $m>1$, come from the fact that the location card $i$ is put into every time it is chosen with the left hand is chosen independently and uniformly on $[n]$ and that $\hat L_i^1(m-1), \ldots, \hat L_i^n(m-1)$ is a bijection of $[n]$. Taking $\hat\tau_i(1)=0$, we have that \begin{align}
 \P\Big(\sum_{m=1}^x\tau_i(m)\le t\Big)\le \P\Big(\sum_{m=1}^x\hat\tau_i(m)\le t\Big)\label{tauhat}
 .
 \end{align}

  It follows that for $m>1$,
\begin{align}\notag
\E(\hat\tau_i(m)) &= \sum_{h=1}^n \frac1{n}\E(\hat\tau_i(m)\,|\, H=h)= \frac1{n} \sum_{h=1}^n \sum_{r=1}^n \E(\tilde\tau_i(m)\,|\, H=h, R=r)\P(R=r)\\ \notag
&=\frac1{n} \sum_{h=1}^n \sum_{r=1}^n \E(\hat\tau_i(m)\,|\, H=h, R=r)\left(1-1/n\right)^{r-1}\frac1{n}\\ \notag
&=\frac1{n^2} \sum_{h=1}^n \Big(\sum_{r=1}^{h-1}r\left(1-1/n\right)^{r-1}+\sum_{r=h}^{n}h\left(1-1/n\right)^{r-1}\Big)\\
&=\frac1{2}(n-3)(1-1/n)^n+1\ge k_1n, \label{exphattau}
\end{align}
for some constant $k_1$.

Similarly, for each $m>1$,
\begin{align*}
\E((\hat\tau_i(m))^2) &= \frac1{n^2} \sum_{h=1}^n \Big(\sum_{r=1}^{h-1}r^2\left(1-1/n\right)^{r-1}+\sum_{r=h}^{n}h^2\left(1-1/n\right)^{r-1}\Big)\\
&\le k_2n^2,
\end{align*}
for some constant $k_2$. Using this and equation (\ref{exphattau}), we have that $\var(\hat\tau_i(m))=k_3 n^2$, for some constant $k_3$. Since the $\hat\tau_i(m)$ are independent for all $m\ge1$, we deduce that \begin{align}\var(\sum_{m=1}^x \hat\tau_i(m))=k_3(x-1)n^2.\label{vartau}\end{align}
From equations (\ref{Nit}) and (\ref{tauhat}), we have that \begin{align}\notag
\E\Big(\sum_{s=1}^t \indic{\exists i\in[k]:\, L_s=\sigma_{s-}^{-1}(c_i)}\Big)&=\sum_{i=1}^k\E(N_i(t))=\sum_{i=1}^k\sum_{x=1}^t\P(N_i(t)\ge x)\\
&=\sum_{i=1}^k\sum_{x=1}^t \P\Big(\sum_{m=1}^x\tau_i(m)\le t\Big)\le \sum_{i=1}^k\sum_{x=1}^t \P\Big(\sum_{m=1}^x\hat\tau_i(m)\le t\Big) \label{sumsum}
.
\end{align}
By Chebyshev's inequality and equation (\ref{vartau}), for any $\alpha>0$, $$
\P\Big(\Big|\sum_{m=1}^x\hat\tau_i(m)-\E\Big(\sum_{m=1}^x\hat\tau_i(m)\Big)\Big|\ge \alpha n\Big)\le \frac{k_3x}{\alpha^2}
.$$
Using equation (\ref{exphattau}) this gives us
$$
\P\Big( \sum_{m=1}^x\hat\tau_i(m)\le (k_1(x-1)-\alpha)n\Big)\le\frac{k_3x}{\alpha^2}
,$$and thus for $t<k_1(x-1)n$,
$$
\P\Big( \sum_{m=1}^x\hat\tau_i(m)\le t\Big)\le \frac{k_3x}{(k_1(x-1)-t/n)^2}
.$$
We now have \begin{align*}
\sum_{x=1}^t \P\Big(\sum_{m=1}^x\hat\tau_i(m)\le t\Big)&=\sum_{x=1}^{\lceil 1+t/(k_1n)\rceil} \P\Big(\sum_{m=1}^x\hat\tau_i(m)\le t\Big)+\sum_{x=\lceil 1+t/(k_1n)\rceil +1}^{t}\P\Big(\sum_{m=1}^x\hat\tau_i(m)\le t\Big)\\
&\le 2+t/(k_1n) + \sum_{x=\lceil 1+t/(k_1n)\rceil +1}^{t}\P\Big(\sum_{m=1}^x\hat\tau_i(m)\le t\Big)\\
&\le 2+t/(k_1n) + \sum_{x=\lceil 1+t/(k_1n)\rceil +1}^{t}\frac{k_3x}{(k_1(x-1)-t/n)^2}\\
&\le 2+t/(k_1n) + k_4\log t\\
&\le C\Big(\frac{t}{n}+\log t\Big)
\end{align*}
for some constants $k_4, C$.
Using this and equation (\ref{sumsum}) completes the proof.
\end{proof}

Let $\sigma_t^1,\ldots,\sigma_t^k$ denote the state at time $t$ of $k$ permutations, each evolving independently and by $(\ell_t)_{t\ge0}$-to-random shuffle and starting from a permutation $\sigma_0$. We describe a coupling of these $k$ permutations and a $(k+1)\xth$ permutation denoted $\sigma^0_t$ which starts from permutation $\sigma_0$ and also evolves by $(\ell_t)_{t\ge0}$-to-random shuffle (but not independently from the other permutations).  The coupling has the property that initially and for a period of time afterwards, the locations of special card $c_i$ in permutation $\sigma^i$ for each $1\le i\le k$ matches the location of card $c_i$ in permutation $\sigma^0$, see Figure \ref{F:cards}.

\begin{figure}[h!]
\begin{center}
\includegraphics[width=100mm]{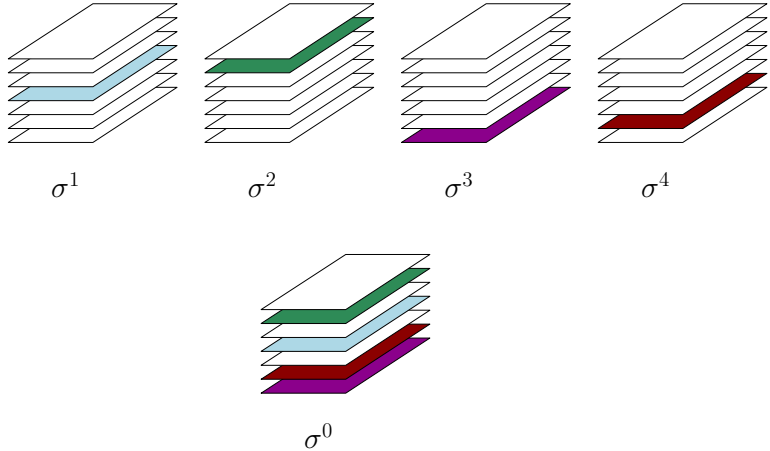} \caption{The desired property of the coupling}\label{F:cards}
\end{center}
\end{figure}

\begin{lemma}\label{L:kcoupling}
There exists a constant $\bar c$ such that for any choice of $k$ special cards, for each $t>0$ and $n$ sufficiently large,\[
\|\mu^{\sigma^0_t}_{c_1,\ldots,c_k}-\mu^{\sigma_t^1}_{c_1}\times\cdots\times\mu^{\sigma_t^k}_{c_k}\|_{\mathrm{TV}}\le \bar c\Big(\frac{tk^2}{n^2}+\frac{k^2\log t}{n}\Big)
.\]
\end{lemma}
\begin{proof}We use a coupling argument to bound this total variation distance.  We use the same choices of $L_1,L_2,\ldots$, for all permutations. For the right hand, we let $R_t^i$ denote the location chosen by the right hand at time $t$ of permutation $\sigma^i$, $1\le i\le k$ (which are independent and uniformly distributed on $[n]$). We describe how to choose $R_t^0$, the choices of the right hand in permutation $\sigma^0$, so that it evolves by $(\ell_t)_{t\ge0}$-to-random with the desired property of Figure \ref{F:cards} for as long as possible.

Firstly, if $L_t$ selects one of the special cards in permutation $\sigma^0$, say $L_t=(\sigma^0_{t-})^{-1}(c_i)$, then we set $R^0_t=R^i_t$. On the other hand, if $L_t$ does not select one of the special $k$ cards, then we toss a coin which lands heads with probability $p$, independently on each toss, where \begin{align}
p:=\frac1{k/n+(1-1/n)^k}\ge1-ck^2/n^2\label{probheads}
,\end{align}for some constant $c$.
If the coin lands tails we choose $U$ uniformly on $\{(\sigma^0_{t-})^{-1}(c_i),\, 1\le i\le k\}$ and set $R_t^0=U$, that is, we choose a special card, each with equal probability.  If the coin lands heads, we choose $U$ uniformly on $[n]$.  If $U$ chooses a special card, that is $U=(\sigma^0_{t-})^{-1}(c_i)$ for some $1\le i\le k$, we set $R^0_t=R_t^i$. If $U$ chooses a non-special card we set $R^0_t=U$ unless at least one of the permutations $\sigma^1,\ldots,\sigma^k$ selects its special card with the right hand, that is \[
E:=\{1\le i\le k:\, R^i_t=(\sigma^i_{t-})^{-1}(c_i)\}\neq \emptyset
.\]In this case we choose $i$ uniformly from the set $E$ and let $R^0_t=R^i_t=(\sigma^i_{t-})^{-1}(c_i)$.

We show that at every step $t$ the choice of the right hand $R^0_t$ is uniformly distributed on $[n]$. Firstly, conditioned on the event $\{ L_t=(\sigma^0_{t-})^{-1}(c_i),\,\text{some }1\le i\le k \}$, then the choice of $R^0_t$ is clearly uniform. We have to check more carefully for the case where we condition on the complement of this event.  We fix non-special card $j$, and calculate the probability that $R^0_t=(\sigma^0_{t-})^{-1}(j)$. We have \begin{align*}
\P\left(R^0_t=(\sigma^0_{t-})^{-1}(j)\right)&=p\Big(\frac{k}{n}\cdot\frac1{n}+\frac1{n-k}\big(1-k/n\big)\big(1-1/n)^k\Big)\\
&=\frac{p}{n}\Big(\frac{k}{n}+\left(1-1/n\right)^k\Big)=\frac1{n}
,\end{align*}as desired. It follows by symmetry that the probability of selecting with the right hand a specific special card is also $1/n$ at every step.

We now let $A_t$ be the event that up to time $t$, the location of card $c_i$ in permutation $\sigma^i$ is equal to the location of card $c_i$ in permutation $\sigma^0$, for all $1\le i\le k$, that is,
\[
A_t:=\{(\sigma^i_{s})^{-1}(c_i)=(\sigma^0_{s})^{-1}(c_i), \,\text{for all }0\le s\le t, \,1\le i\le k\}
.\]
Certainly $A_0$ holds.  We list below the various situations which can result in a mismatch of the locations of the special cards at time $t$. If none of these situations occur by time $t$ then event $A_t$ will hold.

\begin{enumerate}
  \item Suppose $L_t=(\sigma^0_{t-})^{-1}(c_i),\,\text{for some }1\le i\le k$. A mismatch if $R_t^i=(\sigma^i_{t-})^{-1}(c_{i'})$ for some $i'\neq i$, which happens with probability $(k-1)/n,$ or $R_t^{i'}=(\sigma^{i'}_{t-})^{-1}(c_{i'})$ for some $i'\neq i$, which happens with probability $1-(1-1/n)^{k-1}.$
  \item Suppose $L_t$ chooses a non-special card, the coin lands tails and we choose some permutation, say $\sigma^i$, to copy for the right hand. A mismatch if $R_t^i=(\sigma^{i}_{t-})^{-1}(c_{i'})$ for some $i'\neq i$, which happens with probability $(k-1)/n,$ or $R_t^{i'}=(\sigma^{i'}_{t-})^{-1}(c_{i'})$ for some $i'\neq i$, which happens with probability $1-(1-1/n)^{k-1}.$
  \item Suppose $L_t$ chooses a non-special card, the coin lands heads and $U$ chooses a special card, say $c_i$.  A mismatch if $R_t^i=(\sigma^{i}_{t-})^{-1}(c_{i'})$ for some $i'\neq i$, which happens with probability ${(k-1)/n,}$ or $R_t^{i'}=(\sigma^{i'}_{t-})^{-1}(c_{i'})$ for some $i'\neq i$, which happens with probability ${1-(1-1/n)^{k-1}.}$
  \item Suppose $L_t$ chooses a non-special card, the coin lands heads and $U$ chooses a non-special card.  A mismatch if $\big|\{1\le i\le k:\,R_t^i=(\sigma^{i}_{t-})^{-1}(c_{i})\}\big|>1$, which happens with probability $1-(1-1/n)^k-(k/n)(1-1/n)^{k-1}$.
\end{enumerate}
For $1\le j\le 4$, we let $E_j(t)$ denote the number of times situation $j$ above occurs by time $t$. We use Lemma \ref{leftatspecial} to control $E_1(t)$. Let $C$ be the constant in Lemma \ref{leftatspecial} and $c$ the constant in equation (\ref{probheads}).  We have \begin{align*}
\E(E_1(t))&\le Ck(t/n+\log t)((k-1)/n+1-(1-1/n)^{k-1})\le 2C\frac{k^2}{n}(t/n+\log t),\\
\E(E_2(t))&\le t(1-p)((k-1)/n+1-(1-1/n)^{k-1})\le 2ct\frac{k^3}{n^3},\\
\E(E_3(t))&\le t\frac{k}{n}((k-1)/n+1-(1-1/n)^{k-1})\le 2t\frac{k^2}{n^2},\\
\E(E_4(t))&\le t(1-(1-1/n)^k-(k/n)(1-1/n)^{k-1})\le t\frac{k^2}{n^2}
.\end{align*}
We therefore have by Markov's inequality, \begin{align*}
\P(A_t)&\ge \P(E_1(t)+E_2(t)+E_3(t)+E_4(t)=0)\\
&\ge 1- \E(E_1(t))-\E(E_2(t))-\E(E_3(t))-\E(E_4(t))\\
&\ge 1-\bar c\Big(\frac{tk^2}{n^2}+\frac{k^2\log t}{n}\Big)
,\end{align*}
for some constant $\bar c$. This completes the proof.

\end{proof}

Putting together Lemmas \ref{onecard} and \ref{L:kcoupling} we are able to prove Theorem \ref{T:main}.

\begin{proof}[Proof of Theorem \ref{T:main}]

Let $\pi_0$ be a uniform permutation. Let $\pi^1_t,\ldots,\pi^k_t$ denote the state at time $t$ of $k$ permutations, each evolving independently by $(\ell_t)_{t\ge0}$-to-random shuffle and starting from permutation $\pi_0$. Further, let $\pi^0_t$ be another permutation at time $t$ evolving by the same shuffle and also starting from $\pi_0$ (however, not evolving independently from the others).

By the triangle inequality for total variation we have \begin{align*}
\|\mu^{\sigma^0_t}_{c_1,\ldots,c_k}-\mu^{\pi_0}_{c_1,\ldots,c_k}\|_{\mathrm{TV}}=
\|\mu^{\sigma^0_t}_{c_1,\ldots,c_k}-\mu^{\pi_t^0}_{c_1,\ldots,c_k}\|_{\mathrm{TV}}\le \,& \|\mu^{\sigma^0_t}_{c_1,\ldots,c_k}-\mu^{\sigma^1_t}_{c_1}\times\cdots\times\mu^{\sigma^k_t}_{c_k}\|\tv\\&+
\|\mu^{\sigma^1_t}_{c_1}\times\cdots\times\mu^{\sigma^k_t}_{c_k}-\mu^{\pi^1_t}_{c_1}\times\cdots\times\mu^{\pi^k_t}_{c_k}\|\tv\\&+
\|\mu^{\pi^1_t}_{c_1}\times\cdots\times\mu^{\pi^k_t}_{c_k}-\mu^{\pi_t^0}_{c_1,\ldots,c_k}\|\tv
.\end{align*}
For the first and last inequalities we use Lemma \ref{L:kcoupling} to give
\begin{align*}
\|\mu^{\sigma^0_t}_{c_1,\ldots,c_k}-\mu^{\pi_0}_{c_1,\ldots,c_k}\|_{\mathrm{TV}}&\le 2\bar c\Big(\frac{tk^2}{n^2}+\frac{k^2\log t}{n}\Big)+
\|\mu^{\sigma^1_t}_{c_1}\times\cdots\times\mu^{\sigma^k_t}_{c_k}-\mu^{\pi^1_t}_{c_1}\times\cdots\times\mu^{\pi^k_t}_{c_k}\|\tv\\
&\le 2\bar c\Big(\frac{tk^2}{n^2}+\frac{k^2\log t}{n}\Big)+k\max_{i\in\{c_1,\ldots,c_k\}}\| \mu_i^{\sigma_t^1}-\mu_i^{\pi_0}\|_{\mathrm{TV}}.
\end{align*}
Maximizing over the choice of the $k$ special cards and noticing that the first term on the right-hand side is $o(1)$, we obtain for all $n$ sufficiently large,\[
\max_{(c_1,\ldots,c_k)\in \Omega_k}\|\mu^{\sigma^0_t}_{c_1,\ldots,c_k}-\mu^{\pi_0}_{c_1,\ldots,c_k}\|_{\mathrm{TV}}\le \delta + k\max_{i\in [n]}\| \mu_i^{\sigma_t^1}-\mu_i^{\pi_0}\|_{\mathrm{TV}}
.\]
This holds for any $\sigma_0$ and thus as $n\to\infty$,
\begin{align}
  \max_{(c_1,\ldots,c_k)\in \Omega_k}d_{c_1,\ldots,c_k}(t)\le o(1)+k\max_{i\in[n]}d_i(t).\label{distance}
\end{align}

We deduce that\begin{align*}
\tmix^k(\eps)&:=\min\{t\ge0:\, \max_{(c_1,\ldots,c_k)\in \Omega_k}\|\mu^{\sigma^0_t}_{c_1,\ldots,c_k}-\mu^{\pi_0}_{c_1,\ldots,c_k}\|_{\mathrm{TV}}<\eps\}\\&\le
\min\{t\ge0:\, \delta+k\max_{i\in [n]}\| \mu_i^{\sigma_t^1}-\mu_i^{\pi_0}\|_{\mathrm{TV}}< \eps\}\\
&=\min\{t\ge0:\, \max_{i\in [n]}\| \mu_i^{\sigma_t^1}-\mu_i^{\pi_0}\|_{\mathrm{TV}}< (\eps-\delta)/k\}\\
&=\tmix^1((\eps-\delta)/k)
.\end{align*}
We apply Lemma \ref{onecard} which says that $\tmix^1(\eps)\le -n\log\eps$ to deduce that $\tmix^1((\eps-\delta)/k)\le -n\log((\eps-\delta)/k)$. Thus $\tmix^k(1/4)\le n(\log k-\log(1/4-\delta))$ and taking $\delta=1/4-e^{-3/2}$ gives the desired result that $\tmix^k(1/4)\le n(\log k+3/2)$.
 \qedhere
%
\end{proof}

\begin{rmk}
  The main idea of this proof is to show that the movement of the $k$ special cards is close to independent.  Our restriction on the value of $k$ comes into play here -- for $k$ larger than $n^{1/2}$ their movement will in fact no longer be close to independent (in the sense that there will be times at which the left and the right hands choose a special card at the same time).
\end{rmk}

\section{Cutoff for top-to-random and random-to-random}\label{S:cutoff}
We begin this section by showing cutoff of the partial mixing time of the top-to-random transposition shuffle. The lower bound is essentially the coupon-collector problem.
\begin{proof}[Proof of Theorem \ref{T:top}]

We shall show that for the top-to-random transposition shuffle the following conditions hold:
\begin{alignat*}{2}
  &(i).\qquad &&\limsup_{n\to\infty}\max_{(c_1,\ldots,c_k)\in \Omega_k}d_{c_1,\ldots,c_k}(n\log k+\alpha n)\le e^{-\alpha}.\qquad\phantom{aaaaaaaaaaaaaaaaaaaaaaaaaaaa}\\
  &(ii).\qquad &&\text{Suppose $k\to\infty$ as $n\to\infty$. Then}\\
  & \qquad &&\lim_{\alpha\to -\infty}\liminf_{n\to\infty}\max_{(c_1,\ldots,c_k)\in \Omega_k}d_{c_1,\ldots,c_k}(n\log k+\alpha n)=1.\\
&(iii).\qquad&&\text{Suppose $k$ is bounded above by constant $K$ for all $n$.  Then}\\
&\qquad &&\liminf_{n\to\infty}\max_{(c_1,\ldots,c_k)\in \Omega_k}d_{c_1,\ldots,c_k}(n\log k+\alpha n)\ge e^{-\alpha}/K.
\end{alignat*}

From equation (\ref{distance}), we have \[
\max_{(c_1,\ldots,c_k)\in \Omega_k}d_{c_1,\ldots,c_k}(n\log k+\alpha n)\le o(1)+k\max_{i\in [n]}d_i(n\log k+\alpha n)
.\]
However, from Lemma \ref{onecard}, \[
\max_{i\in [n]}d_i(n\log k +\alpha n)\le e^{-\alpha}/k
.\]

For statement (ii), we show that for each $\eps>0$,\[
\lim_{\alpha\to -\infty}\liminf_{n\to\infty}\max_{(c_1,\ldots,c_k)\in \Omega_k}d_{c_1,\ldots,c_k}(n\log k+\alpha n)\ge 1-\eps
.\]
Fix $\eps>0$ and let $C=2/\eps$. For a choice of $\cC=\{c_1,\ldots,c_k\}$, let $\cR_t=\{\sigma_{1-}(R_1),\ldots,\sigma_{t-}(R_t)\}$ be the labels of cards chosen by the right hand up to time $t$.
We let $T$ denote the first time that exactly $C$ of the $k$ special cards are yet to be selected with the right hand in the evolution of $\sigma_t$, \[
T=\inf\left\{t\ge0:\,|\cR_t\cap\cC|=k-C\right\}
.\] We let $E(\mu)$ be the event that permutation $\mu\in S_n$ has more than $C$ fixed points.  Since $\{T> t\}\subseteq E(\sigma_{t})$, it suffices to show that $\P(T>n\log k -\alpha n)\ge 1-\eps/2$ and $\P(E(\pi))\le \eps/2$
for $n$ and then $\alpha$ sufficiently large for a uniformly chosen permutation $\pi$.

We denote by $X_t$ the number of special cards that have not been moved by time $t$ in the evolution of $\sigma_t$. We show $X_t$ is concentrated around its mean.  We let $A_i(t)$ be the event that card $c_i$ has not been selected by time $t$.  Then we can write $
X_t=\sum_{i=1}^k\indic{A_i(t)},$ and so $\E(X_t)=k(1-1/n)^t$. Furthermore,
\begin{align*}
  \E(X_t^2)&=\E\Big(\sum_{i=1}^k\indic{A_i(t)}+\sum_{i=1}^k\sum_{j\neq i}\indic{A_i(t)\cap A_j(t)}\Big)\\
&=k(1-1/n)^t+k(k-1)(1-2/n)^t.\\
\intertext{Thus we have}
\var(X_t)&=k(1-1/n)^t+k(k-1)(1-2/n)^t-k^2(1-1/n)^{2t}\\
&<k(1-1/n)^t+k^2\left[(1-2/n)^t-(1-1/n)^{2t}\right]\\
&<k(1-1/n)^t.
\intertext{Recall that $k\to\infty$ as $n\to\infty$.  By Chebyshev's inequality it follows that}
\P(X_t>C)&\ge 1-\P(|X_t-\E(X_t)|>\E(X_t)-C)\\
&\ge 1-\frac{\var(X_t)}{(\E(X_t)-C)^2}\\
&> 1-\frac{k(1-1/n)^t}{(k(1-1/n)^t-C)^2}\\
&\ge 1-\eps/2,
\end{align*}
for all $n$ and then $\alpha$ sufficiently large, with $t=n\log k-\alpha n$. We are left to show that $\P(E(\pi))\le \eps/2$ for all $n$ sufficiently large. However, this follows trivially by Markov's inequality since the expected number of fixed points in a uniformly chosen permutation converges to 1 as $n\to\infty$ and $C=2/\eps$.

For statement (iii), we let $E(\mu)$ be the event that special card $c_1$ is in location $\sigma_0^{-1}(c_1)$ in permutation $\mu$.  Clearly, for a uniformly chosen permutation $\pi$, $E(\pi)=1/n$.  On the other hand, $\P(\sigma_t)\ge \P(\text{Geom}(1/n)>t)=(1-1/n)^t$ and thus \[
\liminf_{n\to\infty}\max_{(c_1,\ldots,c_k)\in \Omega_k}d_{c_1,\ldots,c_k}(n\log k+\alpha n)\ge e^{-\log k-\alpha}\ge e^{-\alpha}/K
.\] It follows that $d(\alpha)$ defined as \[
d(\alpha):=\lim_{n\to\infty}\max_{(c_1,\ldots,c_k)\in \Omega_k}d_{c_1,\ldots,c_k}(n\log k+\alpha n)
,\]
lies somewhere in the shaded region of Figure \ref{F:nocutoff}.
\end{proof}

\begin{figure}[h!]
\begin{center}
\includegraphics[width=80mm]{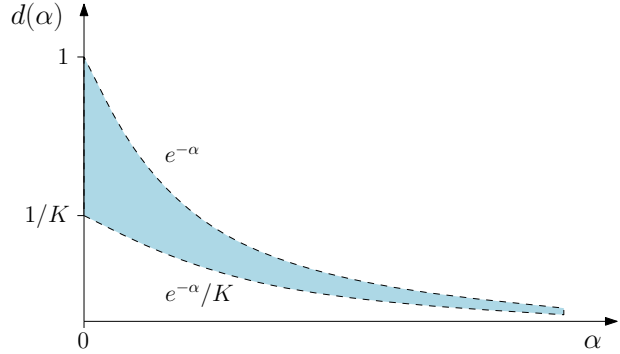} \caption{No cutoff for finite $k$}\label{F:nocutoff}
\end{center}
\end{figure}

We consider now the random-to-random transposition shuffle. We obtain cutoff by improving the upper bound on the mixing time of 1 card and again using coupon-collector arguments for the lower bound.

\begin{proof}[Proof of Theorem \ref{T:random}]

We show that the following hold:
\begin{alignat*}{2}
  &(i).\qquad &&\limsup_{n\to\infty}\max_{(c_1,\ldots,c_k)\in \Omega_k}d_{c_1,\ldots,c_k}(0.5n\log k+\alpha n)\le e^{-2\alpha}.\qquad\phantom{aaaaaaaaaaaaaaaaaaaaaaaaaaaa}\\
  &(ii).\qquad &&\text{Suppose $k\to\infty$ as $n\to\infty$. Then}\\
  & \qquad &&\lim_{\alpha\to -\infty}\liminf_{n\to\infty}\max_{(c_1,\ldots,c_k)\in \Omega_k}d_{c_1,\ldots,c_k}(0.5n\log k+\alpha n)=1.\\
&(iii).\qquad&&\text{Suppose $k$ is bounded above by constant $K$ for all $n$.  Then}\\
&\qquad &&\liminf_{n\to\infty}\max_{(c_1,\ldots,c_k)\in \Omega_k}d_{c_1,\ldots,c_k}(0.5n\log k+\alpha n)\ge e^{-2\alpha}/K.
\end{alignat*}

  For (i) by using similar arguments to the proof in the last lemma, it suffices to show that $\max_{i\in [n]} d_i(t)\le e^{-2t(1-2/n)/n}$.  To do this we adapt the proof of Lemma \ref{onecard}: we couple now both the choices of the left hand and the right hand so that the locations of card $i$ in the two decks will become equal the first time that either $L_t$ or $R_t$ chooses card $i$ in $\pi_t$ and the other chooses a location other than $\pi_t^{-1}$ or $\sigma_t^{-1}$. Note that if both hands make choices from $\{\sigma_{t-}^{-1}(i),\pi_{t-}^{-1}(i)\}$ the cards will not become matched. We thus have \[                                                                                                                                                                                                                                                                                                                                                                                                                                                                                \P(\sigma_{t-}^{-1}(i)=\pi_{t-}^{-1}(i))\ge 1-\exp(-2t(1-2/n)/n).                                                                                                                                                                                                                                                                                                                                                                                                                                                                              \]

For (ii) we use the same argument as in the proof of (ii) in the previous lemma, but with a few modifications. We now set $T=\inf\{t\ge0:\,|(\cR_t\cup\cL_t)\cap\cC|=k-C\}$, where $\cL_t=\{\sigma_{1-}(L_1),\ldots,\sigma_{t-}(L_t)\}$ is the set of labels of cards chosen by the left hand by time $t$. We now obtain $\E(X_t)=k(1-1/n)^{2t}$ and $\E(X_t^2)=k(1-1/n)^{2t}+k(k-1)(1-2/n)^{2t}$. We thus have $\P(X_t>C)\ge 1-\eps/2$ for all $n$ and then $\alpha$ sufficiently large, with $t=0.5n\log k-\alpha n$. This complete the proof of cutoff of the partial mixing time of the random-to-random transposition shuffle.

For statement (iii), we again let $E(\mu)$ be the event that special card $c_1$ is in location $\sigma_0^{-1}(c_1)$ in permutation $\mu$.  As before we have $E(\pi)=1/n$.  On the other hand, \[\P(\sigma_t)\ge \P(\text{Geom}(1/n)>t)^2=(1-1/n)^{2t},\] and thus \[
\liminf_{n\to\infty}\max_{(c_1,\ldots,c_k)\in \Omega_k}d_{c_1,\ldots,c_k}(0.5n\log k+\alpha n)\ge e^{-\log k-2\alpha}\ge e^{-2\alpha}/K
.\] \qedhere
\end{proof}

\section{Cyclic-to-random upper bound}\label{S:cyclic}
We now consider the mixing time of one card in a deck which evolves by the cyclic-to-random transposition shuffle.  Using the coupling technique for the choices of the right hand from Lemma \ref{onecard}, when card $i$ is chosen in deck $(\pi_t)$ with the right hand, we select card $i$ in deck $(\sigma_t)$ with the right hand so that after this transposition cards $i$ will be in the same location in their respective decks. We shall also make other modifications to the choice of the right hand to speed up the time to couple the cards with label $i$ (we shall refer to them as the $i$ cards).

We denote by $R_t^\sigma$ and $R_t^\pi$ the locations chosen by the right hand at time $t$ in decks $(\sigma_t)$ and $(\pi_t)$, respectively.
The left hand chooses the same locations in each deck at every time.  Suppose at time $t$, $L_t$ selects position $\sigma_{t-}^{-1}(i)$ (i.e. card $i$). Further suppose that in a few steps at time $s$ the left hand will select location $\pi_{t-}^{-1}(i)$.  A good choice for $R^\pi_s$ would therefore be $R^\sigma_t$ since card $i$ in $\sigma_s$ will likely still be in position $R^\sigma_t$. We formalize this idea in the proof of the following lemma.

\begin{lemma}\label{onecardcyclic}
 For the cyclic-to-random transposition shuffle, there exists a constant $c>0$ such that for all $n$ sufficiently large we have \[
\max_{i\in[n]} d_i(t) \le c\exp(-t/n)((0.237)^{\lfloor 0.693t/n\rfloor}+1/n).
\]
\end{lemma}
\begin{proof}
We shall use the notation in which $\sigma(i)=j$ says that card $j$ is in position $i$, so that $\sigma(1)$ is the label of the top card in the deck.
  Let $\cC_t$ be the event that the locations of the cards with label $i$ in each deck are equal at time t, that is:
  \[
  \cC_t:=\{\sigma_t^{-1}(i)=\pi_t^{-1}(i)\}
  .\]

Let $\cR_t$ be the event that by time $t$ the right hand has selected card $i$ in $\pi$, that is:
  \[
  \cR_t:=\{\exists \,s\le t:\, R_{s-}^\pi=\pi_{s-}^{-1}(i)\}
  .\]
  If $R^\pi_s=\pi^{-1}_{s-}(i)$ we set $R^\sigma_s=\sigma^{-1}_{s-}(i)$ so that $\cR_t\subset\cC_t$ and therefore \begin{align}\label{C}
  \P(\cC^\complement_t)=\P(\cC^\complement_t|\,\cR_t^\complement)\P(\cR_t^\complement)=\P(\cC^\complement_t|\,\cR_t^\complement)\exp(-t/n)
  .\end{align}
We are interested in matching the locations of the two $i$ cards as quickly as possible. One property of the cyclic-to-random shuffle which we shall exploit is that if a card is not selected during a round ($n$ steps) with the right hand, then it will be selected with the left hand at least once in that round.

For each $t$, we define a random variable $\delta_t$ which takes value either 0 or 1 and which we shall refer to as the \emph{phase} of the system at time $t$.  At time $0$ we set $\delta_0=1$.
Let $M_t(i)=\max(\sigma_t^{-1}(i),\pi_t^{-1}(i))$ and $m_t(i)=\min(\sigma_t^{-1}(i),\pi_t^{-1}(i))$.  We shall define a notion of distance denoted $D_t$ between the location of card $i$ in $\sigma_t$ and the location of card $i$ in $\pi_t$. For times $t$ of phase 1, we define this distance $D_t$ to be
\begin{align*}
  D_t:=
\begin{cases}
  n-M_t(i)+m_t(i),& \text{if $\,m_t(i)<L_t\le M_t(i)$,}\\
M_t(i)-m_t(i), & \text{otherwise.}
\end{cases}
\end{align*}
Let $\eps\in(0,1/2)$ be a constant, which will be chosen later.  If at time $t$, $0<D_t\le\eps n$ we shall say the $i$ cards are \emph{close}, otherwise they are \emph{far}.
If we are in phase 1 at time t, we will enter phase 0 at time\[
\min\{s>t:\, L_s\in\{\sigma_{t-}^{-1}(i),\pi_{t-}^{-1}(i)\}\,\mathrm{ and }\, 0<D_s\le \eps n\}
.\]
This is the first time after time $s$ that the left hand selects an $i$ card when the two $i$ cards are close.

If at time $t$ we are in phase 0, which we entered at time $s<t$, we will leave it (and return to phase 1)  at time\[
\min\{r>t:\, L_r=\{\sigma_{s-}^{-1}(i),\pi_{s-}^{-1}(i)\}\setminus \{L_s\}\}
.\]
The distance $D_t$ will remain constant during times of phase 0.

The proof of this lemma uses three different coupling algorithms. The first is a coupling of the evolution of the two decks of cards, the second is a coupling of these two decks with a certain Markov chain, and the final coupling is of this Markov chain with another, simpler, Markov chain.

We first describe the coupling of the two decks of cards. For each $t\in\mathbb{N}$, let $V_t$ be an independent Bernoulli random variable with $\P(V_t=1)=1/n$.  We denote by $\tau_i$ the time of the start of the $i\xth$ phase 0. We shall define a permutation $\mu_m$ as
\[\mu_m:=\begin{cases}
\sigma_{\tau_m-},&\mbox{if }\indic{\sigma_{\tau_m-}(L_{\tau_m})=i}=1\\
\pi_{\tau_m-},&\mbox{otherwise.}
\end{cases}\]We set $\nu_m:=\{\sigma_{\tau_m},\pi_{\tau_m}\}\setminus \mu_m$.
Further, we define\[
R^\mu_{\tau_m}=\begin{cases}
  R^\sigma_{\tau_m},&\mbox{if }\mu_m=\sigma_{\tau_m-}\\
  R^\pi_{\tau_m},&\mbox{if }\mu_m=\pi_{\tau_m-},
\end{cases}
\]and similarly for $R^\nu_{\tau_m}$.
The coupling algorithm at time $t$ is as follows:
\begin{itemize}
  \item If there exists $s<t$ such that $V_s=1$, choose $R^\pi_t$ uniformly on $[n]$ and set $R^\sigma_t=R^\pi_t$. (In this situation we will have already matched the locations of the $i$ cards.)
  \item If for all $s<t$, $V_s=0$, but $V_t=1$, set $R^\pi_t=\pi_{t-}^{-1}(i)$ and $R^\sigma_{t}=\sigma_{t-}^{-1}(i)$. (This matches the locations of the $i$ cards.)
  \item If for all $s\le t$, $V_s=0$, \begin{itemize}
                                          \item if $i\notin \{\pi_{t-}(L_t),\sigma_{t-}(L_t)\}$, choose independently $R^\pi_t$ uniformly on $[n]\setminus\pi_{t-}^{-1}(i)$ and $R^\sigma_t$ uniformly on $[n]\setminus\sigma_{t-}^{-1}(i)$.
                                          \item if $i\in \{\pi_{t-}(L_t),\sigma_{t-}(L_t)\}$ and $D_t=0$, set $R_t^\sigma=R_t^\pi$ to be chosen uniformly on $[n]\setminus\sigma^{-1}_{t-}(i)$.
                                          \item if $i\in \{\pi_{t-}(L_t),\sigma_{t-}(L_t)\}$, $D_t\neq0$ and there does not exist an $m$ with $t=\tau_m+D_{\tau_m}$, we choose independently $R_t^\pi$ uniformly on $[n]\setminus\pi^{-1}_{t-}(i)$ and $R_t^\sigma$ uniformly on $[n]\setminus\sigma^{-1}_{t-}(i)$. However, if in fact there exists $\bar m$ with $t=\tau_{\bar m}$ and $R^\mu_t=\nu^{-1}_{\bar m}(i)$, we shall say that the coupling fails and terminate.
                                          \item if $i\in \{\pi_{t-}(L_t),\sigma_{t-}(L_t)\}$, $D_t\neq0$ and there exists $m$ with $t=\tau_m+D_{\tau_m}$, we choose $R^\mu_t$ uniformly on $[n]\setminus R^\mu_{\tau_m}$ and then toss a coin which lands heads with probability $1/(n-1)$.  If it lands heads we say the coupling fails and terminate.  If it lands tails we set $R^\nu_t=R^\mu_{\tau_m}$.
                                        \end{itemize}
\end{itemize}

Note that if at any time we have $D_t=0$ then $D_s=0$ for all $s\ge t$.  Figure \ref{F:cardalg} shows a possible situation in which the locations of the $i$ cards become matched at the end of a phase 0.

\begin{figure}[h!]
\begin{center}
\includegraphics[width=90mm]{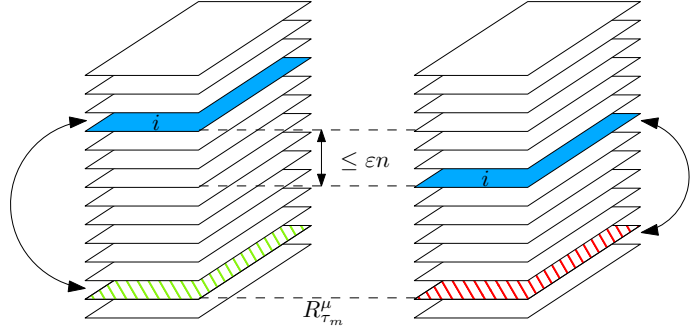} \caption{In this diagram showing how cards can become matched, position 1 is the top of the deck and $n$ is the bottom}\label{F:cardalg}
\end{center}
\end{figure}

It is clear that the distribution of $R^\pi_t$ is independent of $(R^\pi_1,\ldots, R^\pi_{t-1})$ and similarly $R^\sigma_t$ is independent of $(R^\sigma_1,\ldots, R^\sigma_{t-1})$. We now show uniformity. For all times $t$ we have \[\P(\sigma_{t-}(R^\sigma_t)=i)=\P(\pi_{t-}(R^\pi_t)=i)=\P(V_t=1)=1/n.\]

For times $t\notin\{\tau_m+D_{\tau_m},\tau_m\}$ for any $m$ and $j\neq i$, \[\P(\sigma_{t-}(R^\sigma_t)=j)=\P(\pi_{t-}(R^\pi_t)=j)=\Big(\frac1{n-1}\Big)\Big(\frac{n-1}{n}\Big)=1/n.\]  For $t=\tau_m$ for some $m$, and $j\notin\{\mu^{-1}_m(i),\nu^{-1}_m(i)\}$\[
\P(R^\mu_{\tau_m}=j)=\Big(\frac1{n-1}\Big)\Big(\frac{n-1}{n}\Big)=1/n.\] For $t=\tau_m+D_{\tau_m}$ for some $m$, and $j\notin\{\mu^{-1}_m(i),\nu^{-1}_m(i)\}$,\[
\P(R^\nu_t=j)=\Big(\frac{n-1}{n}\Big)\Big(\frac{n-2}{n-1}\Big)\P\left(R^\mu_{\tau_m}=j\,|\,R^\mu_{\tau_m}\notin\{\mu_m^{-1}(i),\nu^{-1}_m(i)\}\right)=1/n.\]
The reason for having the conditional probability in the previous equation is since if there exists such a $t$, we know we have not had either success or failure by this time ($R^\mu_{\tau_m}=\mu^{-1}_m(i)$ iff $V_{\tau_m}=1$ which implies coupling success and $R^\mu_{\tau_m}=\nu_m^{-1}(i)$ implies coupling failure).

For $a,b\in\{1,\ldots,n\}$, define
\[\|a-b\|:=\left\{
  \begin{array}{ll}
    a-b, & \hbox{if $a>b$;} \\
    n-b+a, & \hbox{otherwise.}
  \end{array}
\right.\]

We note that with this definition $\|a-b\|=n-\|b-a\|$.

Let \[\tc=\min\{\min\{t\ge1:\, V_t=1\},\min\{t\ge 1:\, \exists\, m:\tau_m\le t\,\,\mathrm{and}\,\,\|R^\mu_{\tau_m}-L_{\tau_m}\|> D_{\tau_m}\}\},\]be the first time the locations of the $i$ cards become matched.
We now define the coupling of the two decks with a 3-state Markov chain, denoted $(X_t)_{t\ge0}$. This coupling uses the $(R^\sigma_t)_{t\ge1}$ and $(R^\pi_t)_{t\ge1}$ which we have just constructed. This discrete-time Markov chain has state space $\{C,F,S\}$. This will correspond to the $i$ cards being close, far and coupled (success state) respectively. We shall construct the chain so that one step of it corresponds to at most $(1+\eps)n$ steps of the card shuffling. Our interest is in showing that
\[
\P(\cC_t^\complement|\,\cR_t^\complement)\le \P(X_{\lfloor t/(1+\eps)n\rfloor}\neq S)
,\]
for some absorbing state $S$. We note that we only require coupling the evolution of the decks with a Markov chain for times $t$ with event $\cR_t^\complement$ holding. Therefore in the following discussion we assume that this event holds for all times referred to.

We begin the coupling by constructing a sequence of independent random variables denoted $(U_m)_{m\ge1}$, each uniform on $\{1,\ldots,\eps n\}.$ The desired property of these random variables is that for each $m\ge1$ with $\tau_m<\tc$, $U_m\ge D_{\tau_m}$, almost surely.

If $\exists \ell< m$ such that $\|R^\mu_{\tau_\ell}-L_{\tau_\ell}\|>D_{\tau_\ell}$, we simply choose $U_m$ uniformly on $\{1,\ldots,\eps n\}.$


Conditionally on $\tau_m<\tc$, we define three disjoint events, denoted $E^1_m$, $E^2_m$, and $E^3_m$:
\begin{itemize}
  \item $E^1_m:=
\exists\, s\in[\tau_m+\|R^\mu_{\tau_m}-L_{\tau_m}\|,\,\tau_m+\|R^\mu_{\tau_m}-L_{\tau_m}\|+(1-\eps)n),$ such that \mbox{$\mu_s(L_s)=i,$} $R_s^\mu\in[R^\mu_{\tau_m}-\eps n,R^\mu_{\tau_m}).$ We denote by $s_m$ this value of $s$.
  \item $E^2_m:=
\exists\, s\in[\tau_m+\|R^\mu_{\tau_m}-L_{\tau_m}\|,\,\tau_m+\|R^\mu_{\tau_m}-L_{\tau_m}\|+(1-\eps)n),$ such that \mbox{$ \mu_s(L_s)=i$}, $R^\mu_s\in[R^\mu_{\tau_m},\min(R^\mu_{\tau_m}+s-\tau_m-\|R^\mu_{\tau_m}-L_{\tau_m}\|,R^\mu_{\tau_m}+\eps n)].$
We denote by $s_m$ this value of $s$.
  \item $E^3_m:=(E^1_m\cup E^2_m)^\complement $
\end{itemize}

Figure \ref{F:events} shows possible trajectories for events $E^1_m$ and $E^2_m$ (for event $E^2_m$ we show two possible trajectories for one of the $i$ cards).

\begin{figure}[h!]\label{F:events}
\begin{center}
\subfigure{
\includegraphics[width=65mm]{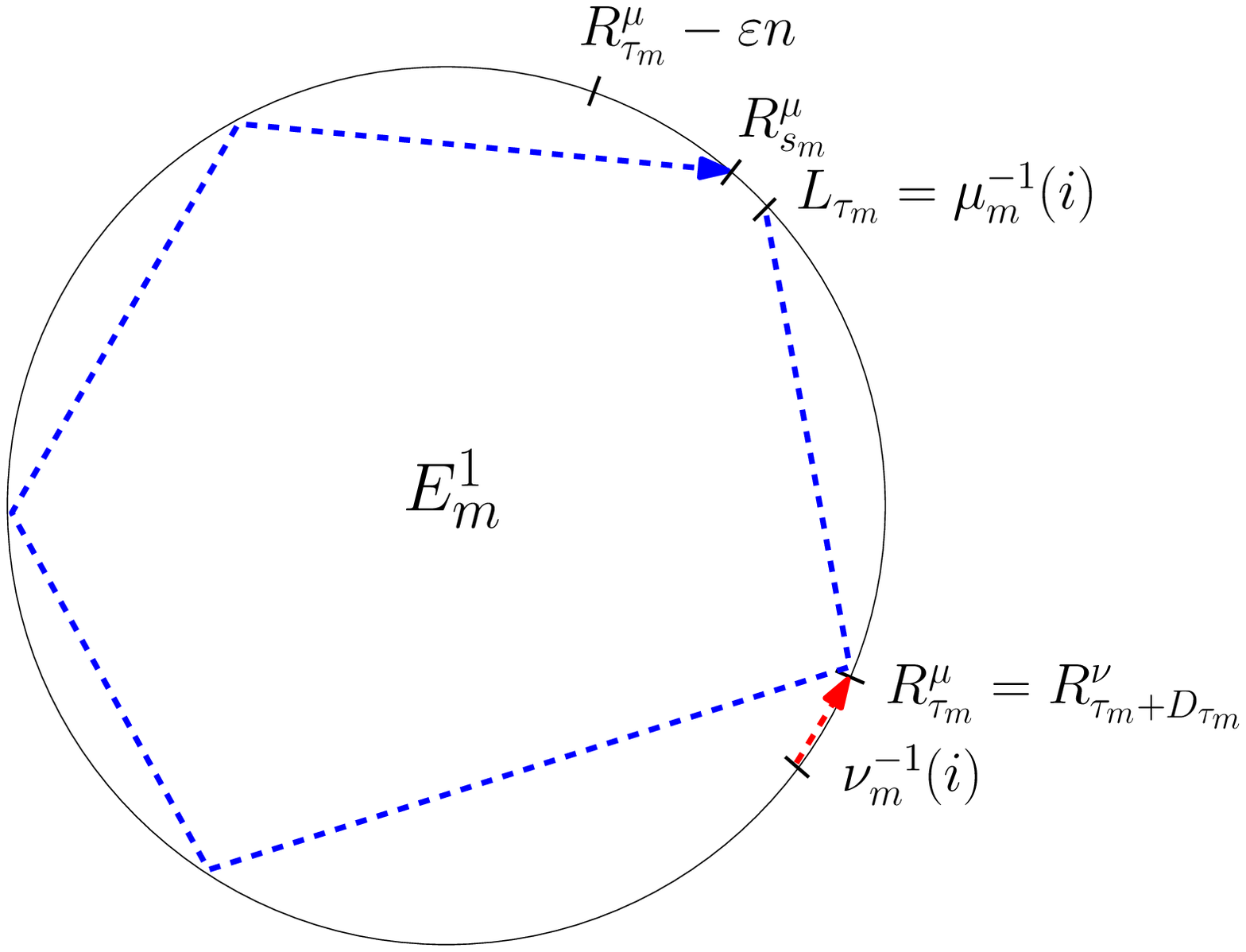}
}
\subfigure{
\includegraphics[width=65mm]{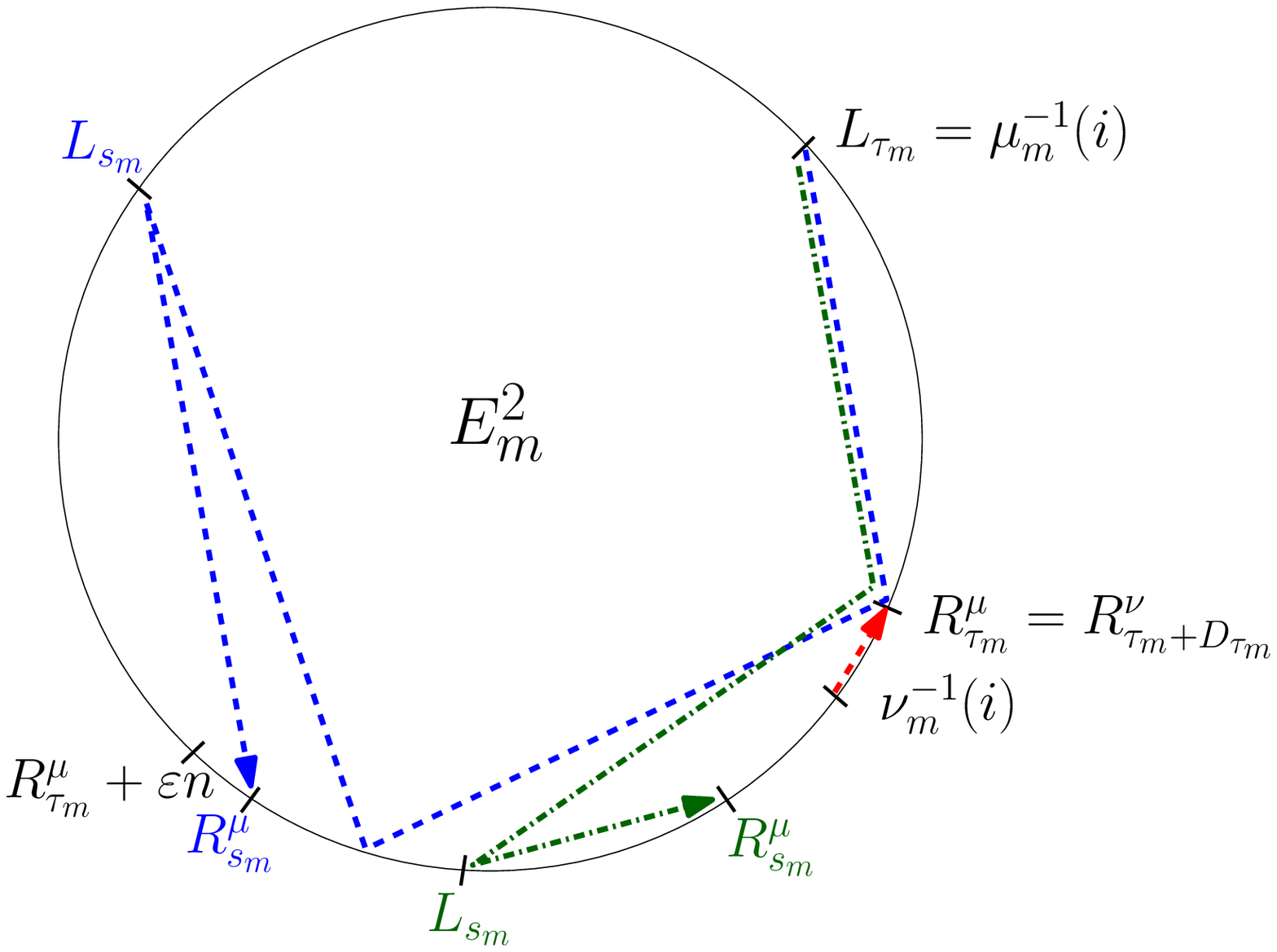}
}
\caption{In these diagrams showing card trajectories, the left hand cycles clockwise around the circles}
\end{center}
\end{figure}

If, conditionally on $\tau_m<\tc$, event $E^1_m$ holds for some $s_m$, set $U_{m+1}=\| R^\mu_{\tau_m}-R^\mu_{s_m}\|=D_{\tau_{m+1}}$. If however, event $E^2_m$ holds we set $U_{m+1}$ equal to $\|R_{s_m}^\mu-R^\mu_{\tau_m}\|$ with probability $\alpha/(\eps n)$ and with probability \mbox{$1-\alpha/(\eps n)$} we choose $U_{m+1}$ uniformly from the set $\{\alpha+1,\ldots,\eps n\}$, where \mbox{$\alpha=\min(\|L_{s_m}-R^\mu_{\tau_m}\|,\eps n)$.} If $E^3_m$ holds we set $U_{m+1}=D_{\tau_{m+1}}$.

We have to check that under this construction of the $U_m$, they are all independent, uniform on $[\eps n]$, and satisfy $U_m\ge D_{\tau_m}$ almost surely. For the independence, we note that the construction of $U_{m+1}$ is independent of any of the previous distances, $D_1,\ldots,D_m$, and therefore is independent of $U_1,\ldots,U_m$.  If events $E^1_m$ or $E^3_m$ hold, it is clear that $U_{m+1}$ is uniform on $[\eps n]$ (for event $E^1_m$, this follows from $R^\mu_{s_m}$ being uniform on $[R^\mu_{\tau_m}-\eps n,R^\mu_{\tau_m})$). If event $E^2_m$ holds, uniformity follows from the fact that $\|R^\mu_{s_m}-R^\mu_{\tau_m}\|$ is uniform on $[\alpha]$ and symmetry. The almost sure inequality $U_{m+1}\ge D_{\tau_{m+1}}$ is clear since $D_{\tau_{m+1}}=\|R_{s_m}^\mu-R^\mu_{\tau_m}\|$.

Next we note that the probability of event $E^1_m\cup E^2_m$ does not depend on the value of $D_{\tau_m}$.  We calculate the probability of this union.

Let $p_s=\P(E^1_m\cup E^2_m\,|\,\mu_s(L_s)=i)$. We are interested in calculating $p_{\tau_m+\|R^\mu_{\tau_m}-L_{\tau_m}\|}$. Note that our conditioning on $\cR^\complement_{\tau_m+(1+\eps n)}$ means that when an $i$ card is selected with the left hand it is not able to stay in that position but jumps uniformly to all other $n-1$ possible positions.  For simplicity of notation during this calculation we shall subtract $\tau_m+\|R_{\tau_m}^\mu-L_{\tau_m}\|$ from the time index. Thus we are interested in $p_0$.  For $s\in(\eps n, n-\eps n)$, \[
p_s=\frac{2\eps n}{n-1}+\sum_{r=s+1}^{n-\eps n-1}\frac{p_r}{n-1}
.\]
We solve this using $p_{n-\eps n-1}=2\eps n/(n-1)$, to deduce that for $s\in(\eps n, n-\eps n)$, \begin{align}
p_s=2\eps(1-1/n)^{s+\eps n-n}\label{ps}
.\end{align}
For $s\in [0,\eps n]$, \[
p_s=\frac{\eps n}{m-1}+\frac{s-1}{n-1}+\sum_{r=s+1}^{\eps n}\frac{p_r}{n-1}+\sum_{r=\eps n}^{n-\eps n-1}2\eps (1-1/n)^{r+\eps n -n}
.\]Using equation (\ref{ps}) we deduce that for $s\in [0,\eps n]$, \[
p_s=1+2\eps(1-1/n)^{s+\eps n -n}-(1-1/n)^{s-1-\eps n},\] and so, in particular, \[
p_0=1+2\eps(1-1/n)^{\eps n-n}-(1-1/n)^{-\eps n-1}
.\]

We now describe how the Markov chain jumps. We shall embed the times of the jumps of the Markov chain into the time of the shuffling process.  This will enable us to prove (by induction) that at time $t$ of the card shuffling, at least $\lfloor t/((1+\eps)n)\rfloor$ jumps of the Markov chain have been made. Furthermore, our construction of the coupling (of the Markov chain with the card shuffling) ensures that if the locations of the two $i$ cards have not become matched by time $t$ then the Markov chain will not be at state $S$ after $\lfloor t/((1+\eps)n)\rfloor$ jumps. Our embedding of the jump-times is such that the chain can only jump at times $t$ such that $L_t\in\{\sigma_{t-}^{-1}(i), \pi^{-1}_{t-}(i)\}$. To begin the inductive process, consider the first time the left hand selects an $i$ card. If the cards are close at this time, we start the Markov chain from state $C$ at this time, otherwise we start the Markov chain from state $F$.

Let \[
M_1:=\inf\{m\ge1:\,\|R^\mu_{\tau_m}-L_{\tau_m}\|>D_{\tau_m}\},\quad M_2:=\inf\{m\ge1:\,\|R^\mu_{\tau_m}-L_{\tau_m}\|>U_m\}
.\]

We note that if $\|R^\mu_{\tau_m}-L_{\tau_m}\|=D_{\tau_m}$, we terminate the couplings at time $\tau_m$ and say they have failed.  Also, note that $M_1\le M_2$ almost surely.
Suppose that the chain has just jumped to state $C$ (either from state $C$ or state $F$). This corresponds to a time $\tau_m$ (start of the $m\xth$ phase 0). If we have $m=M_2$, the chain will next jump to state $S$ at time $\tau_m+\eps n$.  At this time the locations of the two $i$ cards will certainly be matched. If $m=M_1<M_2$, we run the Markov chain independently from the evolution of the cards from time $\tau_m$, according to the transition matrix $P$ given below with jumps every $(1+\eps)n$ steps and with the first jump (at time $\tau_m+(1+\eps)n$) conditioned on not going to state $S$. If $m<M_1$, we are in one of the following three situations:
\begin{enumerate}
  \item $E^1_m$ occurs. Then the next phase 0 starts at time $$\tau_{m+1}=\tau_m+\|R^\mu_{\tau_m}-L_{\tau_m}\|+\|R^\mu_{s_m}-R^\mu_{\tau_m}\|\le \tau_m+\eps n +n.$$ We make the chain jump to state $C$ at this time.
  \item $E^2_m$ occurs. Then the next phase 0 starts at time $$\tau_{m+1}=\tau_m+n+\|R^\mu_{\tau_m}-L_{\tau_m}\|\le\tau_m+n+\eps n.$$ We make the chain jump to state $C$ at this time.
  \item $E^3_m$ occurs.  Then the next jump of the chain is at time $$\inf\{t>\tau_m+D_{\tau_m}:\, L_t\in\{\sigma^{-1}_{t-}(i),\pi^{-1}_{t-}(i)\}\},$$ and it jumps to state $F$. Note that this $t$ satisfies $t\le \tau_m+n$.
\end{enumerate}
In each of these situations the amount of time we have to wait after time $\tau_m$ until the next jump of the Markov chain is less than $(1+\eps) n$.

On the other hand, suppose the chain has just moved to state $F$ (either from state $F$ or from state $C$).  At this time the left hand selects an $i$ card.  If it is moved to a location so that the two $i$ cards are close, we wait until the next time an $i$ card is selected (this will be the start of the next phase 0 and will be less than $n$ steps later) and at this time make the Markov chain jump to state $C$. On the other hand, if it is not moved to such a location, we make the chain jump to state $F$ the next time an $i$ card is selected with the left hand (again this will be less than $n$ steps later). This completes the inductive step, showing that each step of the Markov chain takes at most $(1+\eps)n$ time steps of the shuffling.

%

The fact that this does indeed give a process which is Markov follows from the $U_m$ being independent, identically distributed random variables and the event $E^1_m\cup E^2_m$ being independent from the process up to time $\tau_m$.

Consider time $\tau_m$: the start of the $m\xth$ phase 0, for $m\le M_1$. Using the uniformity of $U_m$, the probability the chain next jumps to state $S$ is\[
\P(\|R^\mu_{\tau_m}-L_{\tau_m}\|>U_m)=\sum_{x=1}^{\eps n} \frac1{\eps n}\P(\|R^\mu_{\tau_m}-L_{\tau_m}\|>x)=
1-\frac{\eps n +1}{2(n-1)}
.\]
Using $p_0$ as calculated above we further deduce that (if currently at state $C$) the probability the chain next jumps to state $C$ is\[
\frac{\eps n+1}{2(n-1)}\left(1-(1-1/n)^{-\eps n-1}+2\eps (1-1/n)^{\eps n-n}\right)
.\]This leaves us with a probability of \[
\frac{\eps n+1}{2(n-1)}\left((1-1/n)^{-\eps n-1}-2\eps (1-1/n)^{\eps n-n}\right)
\]
that the chain jumps to $F$ if at state $C$.

Let $g(n)=1-1/n$.  We obtain the following transition matrix of the Markov chain $(X_m)_{m\ge 1}$ (here the first row/column corresponds to state $C$, the second to state $F$ and the third to state $S$):\[
P:=\left(
     \begin{array}{ccc}
       \frac{\eps n+1}{2(n-1)}\left(1-g(n)^{-\eps n-1}+2\eps g(n)^{\eps n-n}\right) & \frac{\eps n+1}{2(n-1)}\left(g(n)^{-\eps n-1}-2\eps g(n)^{\eps n-n}\right) & 1- \frac{\eps n+1}{2(n-1)}\\
       \frac{2\eps n}{n-1} & 1- \frac{2\eps n}{n-1}& 0\\
       0& 0 & 1\\
     \end{array}
   \right)
.\]

The probability of the chain jumping from $C$ to $S$ includes the possibility of termination of the couplings.  We now fix a $\xi>0$ and note that for all $n>n_0$ for some sufficiently large $n_0$ we have both \begin{align}
\frac{\eps n +1}{2(n-1)}<\frac{\eps}{2}+\xi \label{E:cond1}
\end{align}

and \begin{align}
(1-1/n)^{-\eps n-1}-2\eps (1-1/n)^{\eps n-n}<\left(e^\eps-2\eps e^{-\eps +1}\right)+\xi.\label{E:cond2}\end{align}

We define $\tilde P$ to be the following matrix:\[
\tilde P:=\left(
            \begin{array}{ccc}
              \frac{\eps}{2}\left(1-e^\eps+2\eps e^{1-\eps} +\xi\right) & \frac{\eps}{2}\left(e^\eps-2\eps e^{1-\eps}\right) +\xi(1-\eps/2)& 1-\frac{\eps}{2}-\xi \\
              2\eps & 1-2\eps & 0 \\
              0 & 0 & 1 \\
            \end{array}
          \right)
.\]

Note that we have $P(F,F)<\tilde P(F,F)$ and $P(F,C)>\tilde P(F,C)$. Furthermore, conditions (\ref{E:cond1}) and (\ref{E:cond2}) imply that $P(C,C)>\tilde P(C,C)$, $P(C,S)>\tilde P(C,S)$ and $P(C,F)<\tilde P(C,F)$.  Therefore, performing the obvious coupling between Markov chain $(X_m)_{m\ge1}$ and a Markov chain $(\tilde X_m)_{m\ge1}$ with transition matrix $\tilde P$ we have for all $n>n_0$, and each $m$, $\P(X_m\neq S)\le \P(\tilde X_m\neq S)$. At this point we choose the value of $\eps$ which minimizes the second largest eigenvalue of matrix $\tilde P$ (we take $\xi$ arbitrarily close to 0). We find (numerically) that this optimal value is about 0.442. With this value of $\eps$ we obtain the value $\lambda$ of the second largest eigenvalue to be approximately 0.237. We deduce that there exists a constant $\kappa$ such that for all $n$ sufficiently large (and regardless of starting location), $\P(\tilde X_m\neq S)\le\kappa(0.237)^m$. Finally we calculate the probability of termination occurring before we couple the locations of the two $i$ cards.  Note that termination occurs with probability $2/n$ for every phase 0 that we encounter. On the other hand, we couple the locations of the $i$ cards with probability $1-\eps$ during each phase 0. Therefore the probability that termination occurs before the couplings are successful is $2/(n(1-\eps))$.

We therefore obtain that $\P(\cC^\complement_t|\,\cR^\complement_t)\le c((0.237)^{\lfloor t/1.442n\rfloor}+1/n)$ for some constant $c$ and thus using equation (\ref{C}) we have $\P(\cC^\complement_t)\le c e^{-t/n}((0.237)^{\lfloor 0.693t/n\rfloor}+1/n),$ as required.\qedhere
\end{proof}

\begin{rmk}
 There are several places in which the coupling argument could be strengthened (although this would result in the dynamics becoming more complicated). One of these ways would be to improve on the amount of time we wait between successive times of going from state $F$ to state $F$. Indeed, it is clear that $n$ steps is all that is required rather than $(1+\eps)n$ (since within $n$ steps the left hand will select an $i$ card).
\end{rmk}

\begin{proof}[Proof of Theorem \ref{cyclicupper}]

  We prove this result by combining Lemma \ref{onecardcyclic} and Theorem \ref{T:main}. Taking $\delta=1/4-e^{-3/2}$, we have $\tmix^k(1/4)\le \tmix^1(e^{-3/2}/k)$. We wish to find the smallest $t$ such that $c e^{-t/n}((0.237)^{\lfloor 0.693 t/n\rfloor}+1/n)\le e^{-3/2}/k$. Solving this, we find $t\approx0.5005n(\log k+C)$, with $C=-\log(e^{-3/2}/c-1)$, which completes the proof.\qedhere
\end{proof}

\section{Further work and open problems}\label{S:further}

There are several natural further questions we can ask in relation to partial mixing.  The most interesting question is the following: Does there exist a semi-random transposition shuffle with $k$-partial mixing time of $\alpha n\log k$ (for some $k<n$, $k\to\infty$ as $n\to\infty$) but with mixing time $\beta n\log n$ for some $\beta\neq\alpha$? If this is not the case, then is there a general way to extend partial mixing to the full mixing for any semi-random transposition shuffle? Indeed, it is possible to provide bounds on the partial mixing time for larger values of $k$ using the card marking techniques of \cite{matthews}.

With regards to lower bounds, it can be shown that the partial mixing time of the cyclic-to-random transposition shuffle is $\Theta(n\log k)$ by obtaining a lower bound of approximately $0.12n\log k$ by adapting a method of \cite{peresmoss}. Determining if and when cutoff occurs for both partial and full mixing remains a challenge for this shuffle.

The method we have developed can be applied to other processes not considered here. These include shuffling by $k$-cycles, which is known to have a mixing time of $(n/k)\log n$ (shown by \cite{kcycle}) as well as general interchange processes on graphs. Indeed the top-to-random transposition shuffle is simply the interchange process on the star graph on $n$ vertices, and the random-to-random transposition shuffle is the interchange process on the complete graph on $n$ vertices. Studying the $k$-partial mixing time is then equivalent to placing just $k$ walkers onto different vertices of the graph and asking how long until they reach equilibrium. There exists a class of graphs such that for $k=o((n/\log n)^{1/2})$ we only need to calculate the mixing time of one walker on this graph to determine the $k$-partial mixing time (with the star and complete graphs being in this class).

\emph{Acknowledgements:} The author wishes to thank Nathana\"{e}l Berestycki, Yuval Peres and Perla Sousi for useful discussions.

\bibliographystyle{unsrtnat}
\bibliography{partialbib}

\end{document}